\documentclass[12pt]{amsart}
\newcommand\real{{\mathrm I}\!{\mathrm R} }

\newcommand\rat{{\mathrm Q}\kern-.65em {}^{{}_/ }}

\newtheorem{corollary}{Corollary}
\newtheorem{remark}{Remark}

\newtheorem{theorem}{Theorem}
\newtheorem{lemma}{Lemma}

\begin{document}
\title{Equidistribution results for geodesic flows}
\author{Abdelhamid Amroun}
\maketitle
\address{\begin{center}
Universit\'e Paris Sud, D\'epartement de Math\'ematiques,
CNRS UMR 8628, 91405 Orsay Cedex France
\end{center}}

\begin{abstract} Using  the works of
Ma\~n\'e \cite{Ma} and Paternain \cite{Pat} we study 
the distribution of geodesic arcs with respect to equilibrium
states of the geodesic flow on a closed manifold, equipped with
a $\mathcal{C}^{\infty}$ Riemannian metric.
We prove large deviations lower and upper bounds and a contraction
principle for the geodesic flow in the space of probability measures
of the unit tangent bundle. We deduce a way of approximating
equilibrium states for continuous potentials.
\end{abstract}
\section{Introduction}

Let $M$ be a closed and connected manifold equipped with a 
$C^{\infty }$ Riemannian metric. We study 
the distribution of geodesic arcs of $M$ with respect to equilibrium states.  
We prove large deviations lower and upper bounds for the
geodesic flow in the space of probability measures of the unit
tangent bundle. More precisely, we consider 
Lebesgue measures supported on a finite number of 
geodesic arcs and show that they define a process which satisfy a large
deviation principle with action function given by the topological
pressure. As an application, we obtain equidistribution results 
which describe the proportion of
geodesic arcs which support Lebesgue measures close to equilibrium states.
We show that this proportion converges exponentially fast to
one when the length of the geodesic arcs tends to infinity.
We also prove a contraction principle for these probability
measures, which is a large deviation theorem with constraints.
This work is based on two remarkable formulas due to
Ma\~n\'e \cite{Ma} and Paternain \cite{Pat} which
characterize the topological entropy and pressure as a growth rate
of the number of the geodesic arcs (see Theorem $8$ and Theorem $9$).
On the other hand, the technics of convex analysis and large
deviations in \cite{Ki} and \cite{DZ} were particulary usefull for this work.  
We adapt and extend certain of the general arguments in \cite{Ki} to our
situation. 

We have two important situations where the results of the paper apply.
The first situation concerns the class of manifolds with negative
curvature. In this case, it is well
known that for any H\"older continuous potential there exists a unique
equilibrium state. There are three well known invariant measures in this
setting. The Bowen-Margulis measure, which is the equilibrium state
(a measure of maximal entropy) corresponding to constant
potentials. The harmonic measure which
corresponds to the potential $\frac{d}{dt}|_{t=0}(K\circ
\widetilde{\varphi}^{t})$  
where $K$ is the Poisson kernel and $\widetilde{\varphi}^{t}$ the geodesic flow
of $S\widetilde{M}$, where $\widetilde{M}$ is the universal cover of
the manifold $M$. The  Liouville measure which is the
equilibrium state of the potential $\frac{d}{dt}|_{t=0} \det \left(
d\varphi_{t} |_{E^{s}}\right )$ where $E^{s}$ is the stable tangent
bundle of $SM$ (see \cite{ch} and \cite{che} for more details).
For the ``Liouville potential'', we obtain that the
geodesic arcs are uniformly distributed with respect to the Liouville measure. 

The second situation deals with the more general class of Riemannian manifolds
of nonpositive curvature which are Rank $1$. Then by a result of
Knieper \cite{Kn} we know that there exists a uniquely determined
invariant measure of maximal entropy for the geodesic flow. 
But it is not known up to now which class of raisonable potentials
admit a unique equilibrium state and this question remains open for
Rank $1$ manifolds. However, our results give a way of approximating
equilibrium states of the geodesic flow of such manifolds. Indeed,
these results are applicable everywhere where Ma\~n\'e's and
Paternain's formulas hold.  

Finally, I'm very grateful to Fran\c{c}ois Ledrappier for helpful
conversations.
 
\section{Main results}
\subsection{Preliminaries and notations}
Let $M$ be a closed and connected manifold equipped with a 
$C^{\infty}$ Riemannian metric and $\phi : SM \rightarrow SM$ be the
geodesic flow on the unit tangent bundle $SM$. We assume that $M$ has
volume one, $\int_{M}dx=1$, where $dx$ is the volume form induced by
the Riemannian metric of $M$. 

We denote by
$\mathcal{P}(SM)$ the space of probability measures on $SM$
equipped with the weak star topology. 
Let $\mathcal{P}_{inv}(SM)$ be the
subset of $\mathcal{P}(SM)$ of invariant probability measures. Given a
potential $F\in C_{\real}(SM)$, the topological pressure of $F$ is the
number defined by the variational principle \cite{Wa},
\begin{equation}
P(F)=\sup_{m\in \mathcal{P}_{inv}(SM)}(h(m)+\int_{SM} Fdm),
\end{equation}
where $h(m)$ is the entropy of $m$. For $F=0$ this reduces to
\[
P(0)=\sup_{m\in
  \mathcal{P}_{inv}(SM)}h(m):=h_{top}
\]
where $h_{top}$ is the topological entropy of the geodesic flow. 

An equilibrium state for $F$, is a measure $m\in
\mathcal{P}_{inv}(SM)$ which achieves the maximum in $(1)$,
\[
h(m)+\int_{SM} Fdm=P(F).
\]
We denote by $\mathcal{P}_{e}(F)$ the subset of
$\mathcal{P}_{inv}(SM)$ of equilibrium states corresponding to $F$. By
a result of Newhouse \cite{Ne}, since the metric is $C^{\infty}$, the
entropy map $m\rightarrow h(m)$ is upper semicontinuous. Then $h_{top}<\infty$ 
and consequently, the set $\mathcal{P}_{e}(F)$
is a nonempty closed, compact, convex subset of $\mathcal{P}(SM)$ \cite{Wa}.
 
We define the functional $Q_{F}$ on $C_{\real}(SM)$ based on the
potential $F$ by,
\begin{equation}
Q_{F}(\omega):=P(F+\omega)-P(F).
\end{equation}
By definition, $Q_{F}$ is continuous
on continuous functions (see Lemma $1$ and Remark $1$).
Sometimes we will simply write $Q$, if there is no confusion to be been afraid.

We set for any probability measure $\mu$ on $SM$,
\begin{equation}
J_{F}(\mu):=\sup_{\omega}(\int \omega d\mu- Q_{F}(\omega)),
\end{equation}
where the $\sup$ is taken over the space of continuous functions $\omega$ on
$SM$. Observe that since $Q_{F}(0)=0$, then $J_{F}$ is a non negative
functional and clearly is lower semicontinuous.
We will see that (Lemma $1$ and Remark $1$) that
$\mathcal{P}_{e}(F)=\{J_{F}=0\}$. 
Again, if there is no ambiguity we write $J$ instead of $J_{F}$.

By duality, we have
\begin{equation}
Q_{F}(\omega)=\sup_{\mu \in \mathcal{P}(SM)}(\int \omega d\mu- J_{F}(\mu)).
\end{equation}
For any set $E\subset \mathcal{P}(SM)$ put
\[
J_{F}(E):=\inf_{\mu \in E}J_{F}(\mu).
\]

Given $x$ and $y$ in $M$, we denote by $\gamma
_{xy}:\left[ 0,l(\gamma _{xy})\right] \rightarrow M$ a unit speed geodesic
arc joining $x$ to $y$ with length $l(\gamma _{xy})$.
For any $\delta >0$ and $T>0$ we set, 
\begin{eqnarray*}
G_{\delta,T}(x,y)&:=&\{\gamma _{xy}: T-\delta <l(\gamma _{xy})\leq T \} \
and,\\
G_{T}(x,y)&:=&\{\gamma _{xy}: l(\gamma _{xy})\leq T \}.
\end{eqnarray*}

Recall from \cite{Be} that given
$T>0$, the set of geodesic arcs with length $\leq T$ is finite and
its cardinality is locally constant for an open full Lebesgue measure
subset of $M\times M$ (see also \cite{Pat} p53 for a proof
using Sard's theorem). By a result of Burns K and Gutkin E \cite{BG},
the growth of $\# G_{T}(x,y)$ and a positive 
topological entropy are related to the condition of ``insecurity'' of
the manifold.

The integral $\int_{\gamma _{xy}} \omega$ of a function $\omega$ over the
geodesic arc $\gamma _{xy}$ is defined by
\[
\int_{\gamma _{xy}}\omega:=
\int_{0}^{l(\gamma_{xy})}\omega(\phi_{s}(\dot{\gamma}_{xy}(0))ds. 
\]
Here $\dot{\gamma}_{xy}(0)$ is the initial condition of the geodesic
$\gamma_{xy}$.  
The Lebesgue measure $\delta_{\gamma _{xy}}$ with support in $\gamma
_{xy}$ is now defined by, 
\[
\int_{SM}\omega d\delta_{\gamma
  _{xy}}:=\frac{1}{l(\gamma_{xy})}\int_{\gamma _{xy}}\omega. 
\]
We will write $G_{\delta,T}$ and $G_{T}$ for simplicity, since the
dependence in $(x,y)$ will be always clear. 
\subsection{The results}
Consider a real continuous potential $F$ defined on $SM$. 
For any Borel subset $E$ of $\mathcal{P}(SM)$ we set,
\begin{equation}
\nu_{T}(E):=\frac{\int_{M\times M}(\sum_{\gamma_{xy}\in
    G_{\delta,T}:\delta_{\gamma _{xy}}\in E}e^{\int_{\gamma_{xy}}F})dxdy} 
{\int_{M\times M}(\sum_{\gamma_{xy}\in G_{\delta,T}}
e^{\int_{\gamma_{xy}}F})dxdy}.
\end{equation}
This defines a process $(\nu_{T})_{T>0}$ on the space $\mathcal{P}(SM)$. 
The first result gives large deviations bounds for this process, namely
\begin{eqnarray*}
\limsup_{T\rightarrow +\infty }\frac{1}{T}\log \nu_{T}(K)
&\leq&  -J(K)\\
\liminf_{T\rightarrow +\infty }\frac{1}{T}\log \nu_{T}(O)
&\geq&  -J(O),
\end{eqnarray*} 
respectively, for any closed subset $K$ and open subset $O$ of
$\mathcal{P}(SM)$.

Before we state the main theorem, we do the following assumption under which we
prove the lower bound part of the large deviation theorem. This
condition is well known when we deal with the lower bound part (see
\cite{Ki} and \cite{DZ}).

{\it There exists a countable set $\mathcal{C}:=\{g_{k},
  k\geq 1\}\subset C_{\real}(SM)$ of continuous fucntions such that
  their span is dense in 
 $C_{\real}(SM)$ with respect to the topology of uniforme convergence,
  $\|g_{k}\|=1$ for all $k$, and for all $\beta \in \real^{n}$ the
  potential $\sum_{k=1}^{n}\beta_{k}g_{k}$ has a unique equilibrium state.}

\begin{theorem}
Let $M$ be a closed and connected manifold equipped with a $C^{\infty }$
Riemannian metric and $F\in C_{\real}(SM)$. Then for any $\delta >0$ we have
\begin{enumerate}
\item For any closed subset $K$ of $\mathcal{P}(SM)$
\[
\limsup_{T\rightarrow +\infty }\frac{1}{T}\log \frac{\int \left(
\sum_{\gamma _{xy}\in G_{\delta,T}:
\delta_{\gamma _{xy}}\in K} 
e^{\int_{\gamma _{xy}}F}\right) dxdy}{\int \left( \sum_{\gamma
_{xy} \in G_{\delta,T}}e^{\int_{\gamma
_{xy}}F}\right) dxdy}\leq  -J(K).
\]
\item If for all $\beta \in \real^{n}$ and $g=(g_{1},\ldots,
  g_{n})\in \mathcal{C}^{n}$, $F+\beta \cdot g$ has a unique 
equilibrium state, then for any open subset $O$ of $\mathcal{P}(SM)$,  
\[
\liminf_{T\rightarrow +\infty }\frac{1}{T}\log \frac{\int \left(
\sum_{\gamma _{xy}\in G_{\delta,T}:
\delta_{\gamma _{xy}} \in O}
e^{\int_{\gamma _{xy}}F}\right) dxdy}{\int \left( \sum_{\gamma
_{xy} \in G_{\delta,T}}e^{\int_{\gamma
_{xy}}F}\right) dxdy}
\geq -J(O).
\]
\end{enumerate}
\end{theorem}
As consequence of Theorem $1$ $(1)$ we have the following result which
asserts that the proportion of geodesic arcs supporting a Lebesgue
measure close to an equilibrium state is asymptotically equal to $1$
: $\lim_{T\rightarrow \infty} \nu_{T}(V)=1$ and the convergence is
exponential. Set 
$V^{c}:=\mathcal{P}(SM)\backslash V$.
\begin{corollary}
For any open neighborhood $V$ of $\mathcal{P}_{e}(F)$ in
$\mathcal{P}(SM)$ we have 
\[
\lim_{T\rightarrow +\infty } \frac{\int \left(
\sum_{\gamma _{xy}\in G_{\delta,T}:
\delta_{\gamma _{xy}}\in V}
e^{\int_{\gamma _{xy}}F}\right) dxdy}{\int \left( \sum_{\gamma
_{xy} \in G_{\delta,T}}e^{\int_{\gamma
_{xy}}F}\right) dxdy}=1,
\]where the convergence is exponential with speed $e^{-TJ(V^{c})}$.
\end{corollary}
We know that large deviation principles are preserved under continuous
mapping, this is known as the contraction principle (\cite{DZ}). In the  
present case, the contraction principle reduces the preceeding
theorem to a finite dimensional one. First, we need some notations.

Set for $g\in C_{\real^{n}}(SM)$ and $\alpha \in \real^{n}$, $n>0$,
\[
\mathcal{P}_{g,\alpha}(SM):=\{m\in \mathcal{P}(SM):\int_{SM}gdm=\alpha\}.
\]
We define the functional
\begin{equation*}
J_{g}(\alpha)=\left\{
          \begin{array}{ll}\inf (J(m):m\in \mathcal{P}_{g,\alpha}(SM)&
  \quad \mathrm{if} \quad \mathcal{P}_{g,\alpha}(SM)\ne
  \emptyset \\ 
 +\infty & \quad \mathrm{if} \quad
  \mathcal{P}_{g,\alpha}(SM)=\emptyset \\          
          \end{array}
        \right.
\end{equation*}
and for any $E_{n} \subset \real^{n}$,
\[
J_{g}(E_{n})=\inf \left ( J_{g}(\alpha): \alpha \in E_{n}\right ).
\]
To simplify the notations set $m(g)=\int g dm$ for $m\in \mathcal{P}(SM)$.

\begin{theorem}[Contraction principle] Let $M$ be a closed and
  connected manifold equipped with a $C^{\infty }$ Riemannian metric
  and $F\in C_{\real}(SM)$. Let $g\in C_{\real^{n}}(SM)$. Then for any
  $\delta >0$ we have 
\begin{enumerate}
\item For any closed subset $K_{n} \subset \real^{n}$,
\[
\limsup_{T\rightarrow +\infty }\frac{1}{T}\log \frac{\int \left(
\sum_{\gamma _{xy}\in G_{\delta,T}: 
\delta_{\gamma _{xy}}(g)\in K_{n}}
e^{\int_{\gamma _{xy}}F}\right) dxdy}{\int \left(\sum_{\gamma 
_{xy} \in G_{\delta,T}}e^{\int_{\gamma
_{xy}}F}\right) dxdy}\leq  -J_{g}(K_{n}).
\]
\item  If for all $\beta \in \real^{n}$, $F+\beta \cdot g$ has a unique 
equilibrium state, then for any open subset
  $O_{n}\subset \real^{n}$, 
\[
\liminf_{T\rightarrow +\infty }\frac{1}{T}\log \frac{\int \left(
\sum_{\gamma _{xy}\in G_{\delta,T}:
\delta_{\gamma _{xy}}(g)\in O_{n} }
e^{\int_{\gamma _{xy}}F}\right) dxdy}{\int \left(
  \sum_{\gamma_{xy} \in G_{\delta,T}}
e^{\int_{\gamma_{xy}}F}\right) dxdy}\geq -J_{g}(O_{n}).
\]
\end{enumerate}
\end{theorem}
Part $(1)$ of Theorem $2$ is a consequence of Theorem $1$ $(1)$, by the
continuity of the function $g$.
Note that we do not assume in part ($2$) of Theorem $2$ that $g\in
\mathcal{C}^{n}$. This last condition is used in the proof of
part ($2$) of Theorem $1$.
In other words, the conclusion of Theorem $2$ ($2$) holds for any continuous
$g$ such that for all $\beta \in \real^{n}$, $F+\beta \cdot g$ admits
a unique equilibrium state.

For any $\delta >0$ we define the probability measures on $SM$,
\begin{equation}
m_{\delta,T}:=\frac{\int \left ( \sum_{\gamma _{xy}\in
    G_{\delta,T}}e^{\int_{\gamma_{xy}}F}\cdot
\delta_{\gamma _{xy}}\right )dxdy} 
{\int \left ( \sum_{\gamma _{xy}\in
G_{\delta,T}}e^{\int_{\gamma_{xy}}F}\right )dxdy}.
\end{equation}
As an application of Theorem $1$ ($1$) we prove the following theorem.
\begin{theorem}  Let $M$ be a closed and
  connected manifold equipped with a $C^{\infty }$ Riemannian metric
  and $F\in C_{\real}(SM)$. For any $\delta >0$, any weak limit $m_{\infty}$ of
  $(m_{\delta,T})_{T}$ is an 
  equilibrium states for the potential $F$ , i.e $m_{\infty}$ is
  invariant and satisfies 
\[
h(m_{\infty})+\int F dm_{\infty}=P(F).
\]
\end{theorem}

These theorems apply to the geodesic flow of a manifold of negative
curvature and H\"older continuous potentials $F$. The following
result is particulary important when the unique equilibrium state
$\mu_{F}$ corresponding to $F$ is the Bowen-Margulis measure or the
harmonic measure (section 1). In particular, if $F$ is the
``Liouvile potential'' (section 1) then this corollary says
that the geodesic arcs are uniformly distributed.

\begin{corollary}
Let $M$ be a closed and connected manifold of negative curvature
equipped with a $C^{\infty }$ Riemannian metric. Suppose that the
potential $F$ is H\"older continuous and let $\mu_{F}$ be the unique
corresponding equilibrium state. Then, for any $\delta >0$, the measures
  $(m_{\delta,T})_{T}$ converge weakly  to
$\mu_{F}$ as $T\rightarrow +\infty$.   
\end{corollary}

Consider now the measure $\mu_{max}$ of maximal entropy \cite{Kn} of the
geodesic flow of a rank 1 manifold. Set
\[
\mu_{\delta,T}:=\frac{\int \left ( \sum_{\gamma _{xy}\in
    G_{\delta,T}}\delta_{\gamma _{xy}}\right )dxdy}
{\int \left ( \sum_{\gamma _{xy}\in
G_{\delta,T}}\right )dxdy}.
\]
\begin{corollary}
Let $M$ be a closed and connected rank 1 manifold 
equipped with a $C^{\infty }$ Riemannian metric.
Then, for any $\delta >0$, the measures
$(\mu_{\delta,T})_{T}$ converge weakly to
$\mu_{max}$ as $T\rightarrow +\infty$.  
\end{corollary}

\section{Constant and positive potentials}
We state in this section some results which are not a direct
consequence of the previous ones.
\subsection{Positive potentials}
Consider the probability measures defined on $SM$ by
\[
m_{T}:=\frac{\int \left ( \sum_{\gamma _{xy}\in
    G_{T}}e^{\int_{\gamma_{xy}}F}\delta_{\gamma _{xy}}\right )dxdy}
{\int \left ( \sum_{\gamma _{xy}\in
G_{T}}e^{\int_{\gamma_{xy}}F}\right )dxdy}.
\]
\begin{theorem}  Let $M$ be a closed and connected manifold equipped
with a $C^{\infty }$ Riemannian metric  and $F\in C_{\real}(SM)$.
Assume that $P(F)>0$.
Then, the weak limits of $(m_{T})_{T}$ are
equilibrium states corresponding to the potential $F$. If $M$ has 
negative curvature then $(m_{T})_{T}$ converges to the unique
equilibrium state $\mu_{F}$ corresponding to $F$.
\end{theorem}
Theorem $4$ is a consequence of part $1$ of the following theorem.
\begin{theorem}
Let $M$ be a closed and connected manifold equipped with a $C^{\infty }$
Riemannian metric  and $F\in C_{\real}(SM)$. Then 
\begin{enumerate}
\item If $P(F)>0$ we have for any closed subset $K$ of
  $\mathcal{P}(SM)$ such that $J(K)< P(F)$,
\[
\limsup_{T\rightarrow +\infty }\frac{1}{T}\log \frac{\int \left(
\sum_{\gamma _{xy}\in G_{T}:
\delta_{\gamma _{xy}}\in K}
e^{\int_{\gamma _{xy}}F}\right) dxdy}{\int \left( \sum_{\gamma
_{xy} \in G_{T}}e^{\int_{\gamma
_{xy}}F}\right) dxdy}\leq  -J(K).
\]
\item If for all $\beta \in \real^{n}$ and $g=(g_{1},\ldots,
  g_{n})\in \mathcal{C}^{n}$, $F+\beta \cdot g$ has a unique 
equilibrium state, then for any open subset $O$ of $\mathcal{P}(SM)$,  
\[
\liminf_{T\rightarrow +\infty }\frac{1}{T}\log \frac{\int \left(
\sum_{\gamma _{xy}\in G_{T}:
 \delta_{\gamma _{xy}}\in O} 
e^{\int_{\gamma _{xy}}F}\right) dxdy}{\int \left( \sum_{\gamma
_{xy} \in G_{T}}e^{\int_{\gamma
_{xy}}F}\right) dxdy}
\geq -J(O).
\]
\end{enumerate}
\end{theorem}
Observe that given $F$ and any constant $c$ we have
$\mathcal{P}_{e}(F)=\mathcal{P}_{e}(F+c)$. On the other hand, we can
find a constant $c>0$ such that $F+c>0$. Thus up to a constant
we can always assume that $P(F)>0$. 
\subsection{Constant potentials}
We consider here constant potentials which is equivalent to set
$F\equiv 0$.
\begin{theorem} Suppose that $M$ is a closed and connected manifold equipped
    with a $C^{\infty}$ Riemannian metric. If $M$ has no conjugate points,
    then for any $\delta >0$,
\begin{enumerate}
\item For any closed subset $K$ of $\mathcal{P}(SM)$ and a.e $(x,y)\in
  M\times M$  
\[
\limsup_{T\rightarrow +\infty }\frac{1}{T}\log 
\frac{\#\left\{\gamma _{xy}\in G_{\delta,T}:
 \delta_{\gamma _{xy}}\in K\right\}}
{\# G_{\delta,T}}\leq  -J(K).
\]
\item For a.e $(x,y)\in M\times M$, the weak limits of 
\[
\mu_{\delta,T}(x,y):=\frac{\sum_{\gamma _{xy}\in
    G_{\delta,T}}\delta_{\gamma _{xy}}}
{\# G_{\delta,T}}, 
\]
are measures of maximal entropy.   
\end{enumerate}
\end{theorem}
In part $(1)$ of Theorem $6$, the set of points $(x,y)\in
  M\times M$  for which we have the upper bound depends on the given closed
  set $K$, while in part $(2)$ it depends only on $\delta$.
\begin{theorem} Suppose that $M$ is a closed and connected manifold equipped
with a $C^{\infty}$ Riemannian metric. If $M$ has no conjugate points,
and $h_{top}>0$ then,
\begin{enumerate}
\item For any closed subset $K$ of $\mathcal{P}(SM)$ such that
$J(K)<h_{top}$ and a.e $(x,y)\in M\times M$
\[
\limsup_{T\rightarrow +\infty }\frac{1}{T}\log 
\frac{\#\left\{\gamma _{xy}\in G_{T}:
\delta_{\gamma _{xy}}\in K\right\} 
}{\# G_{T}}\leq  -J(K).
\]
\item For a.e $(x,y)\in M\times M$, the weak limits of 
\[
\mu_{T}(x,y):=\frac{\sum_{\gamma _{xy}\in
    G_{T}}\delta_{\gamma _{xy}}}{\# G_{T}},
\]
are measures of maximal entropy. 
\end{enumerate} 
\end{theorem}

\section{Proofs}
\subsection{Growth of geodesic arcs}
The following two theorems of Ma\~n\'e
\cite{Ma} and Paternain \cite{Pat} are the main tools in the
proof of our results.
\begin{theorem}[R Ma\~n\'e \cite{Ma}] Let $M$ be a closed and
connected manifold equipped with a $C^{\infty }$ metric. Then 
\begin{enumerate}
\item
\[
h_{top}=\lim_{T \rightarrow \infty} \frac{1}{T}\log \int_{M\times M}
\# G_{T}(x,y)dxdy. 
\]
\item If $M$ has no conjugate points, for all $(x,y)$
\[
h_{top}=\lim_{T \rightarrow \infty} \frac{1}{T}\log 
\# G_{T}(x,y). 
\]
\item Suppose that the metric is of class $C^{3}$ and $M$ does not
  have conjugate points. Then for any $\delta >0$ and all $(x,y)$
  we have
\[
h_{top}=\lim_{T \rightarrow \infty} \frac{1}{T}\log 
\# G_{\delta, T}(x,y).
\]
\end{enumerate}
\end{theorem}

\begin{theorem}[G P Paternain \cite{Pat}] Let $M$ be a closed and
  connected manifold equipped with a $C^{\infty }$ metric.
\begin{enumerate}
\item For any $\delta >0$
\[
P(F)=\lim_{T\rightarrow +\infty }\frac{1}{T}\log \int_{M\times M}\left(
\sum_{\gamma _{xy} \in G_{\delta,T}}e^{\int_{\gamma _{xy}}F}\right) dxdy.
\]
\item If $P(F) \geq 0$
\[
P(F)=\lim_{T\rightarrow +\infty }\frac{1}{T}\log \int_{M\times M}\left(
\sum_{\gamma _{xy} \in G_{T}}e^{\int_{\gamma _{xy}}F}\right) dxdy.
\]

\end{enumerate}
\end{theorem}

For any invariant probability measure $\mu$ we set
\begin{equation}
I(\mu):=P(F)-(h(\mu)+\int_{SM} Fd\mu).
\end{equation}
\begin{lemma} 
\begin{enumerate}
\item The functional $Q_{F}$ is convex and continuous on continuous functions. 
\item $Q(\omega)=\sup_{\mu \in \mathcal{P}_{inv}(SM)}(\int \omega
  d\mu-I(\mu))$. In other words, the functionals $I$ and $J_{F}$ agree on
  invariant measures.
\end{enumerate}
\end{lemma}
\begin{proof}
Part $(1)$ is a consequence of the convexity of the pressure function
$P$ and the variational principle $(1)$ from which we can
easily deduce that $|P(f)-P(g)|\leq \|f-g\|_{\infty}$ \cite{Wa}. 
Part $(2)$ follows from $(7)$ and,
\begin{eqnarray*}
&&\sup_{\mu \in \mathcal{P}_{inv}(SM)}(\int \omega  d\mu-I(\mu))\\
&=& \sup_{\mu \in \mathcal{P}_{inv}(SM)}(\int \omega  d\mu-
P(F)+h(\mu)+\int Fd\mu)\\
&=&P(F+\omega)-P(F)=Q_{F}(\omega).
\end{eqnarray*}
\end{proof}

\begin{remark}  $Q_{F}$ is invariant by the
geodesic flow : $Q_{F}(\omega \circ \phi_{t})=Q_{F}(\omega)$, for all $t$ and
continuous function $\omega$. Thus, a probability measure $m$
satisfies $J_{F}(m)=0$ if and only if $m$ is invariant and $I(m)=0$. In
particular, if $K$ is a closed subset of $\mathcal{P}(SM)$ and
$\inf_{m\in K}J_{F}(m)=J_{F}(\mu)$ for some $\mu \in K$, then $\inf_{m\in
K}J_{F}(m)=0$ iff $\mu$ is invariant and $h(\mu)+\int Fd\mu=P(F)$. In
other words we have $\mathcal{P}_{e}(F)=\{J_{F}=0\}$.
\end{remark}
\subsection{Proof of Theorem 1 ($1$)}
\begin{proof}
We have to prove
\[
\limsup_{T\rightarrow \infty}\frac{1}{T}\log \nu_{T}(K) \leq
-\inf_{m\in K}J(m):=-J(K).
\]
Let $\epsilon >0$. There exists a finite number of continuous
functions $\omega_{1}, \cdots, \omega_{l}$ such that $K\subset
\cup_{i=1}^{l}K_{i}$, where 
\[
K_{i}=\{m\in \mathcal{P}(SM): \int \omega_{i}dm-Q(\omega_{i})>J(K)-\epsilon\}.
\]
Put
\[
\Gamma_{i}(x,y,T):=\{\gamma_{xy}\in G_{\delta,T}
:\delta_{\gamma _{xy}}\in K_{i}\}, 
\]and
\[
Z_{i}(T):=\int_{M\times M} \sum_{\gamma_{xy}\in
  \Gamma_{i}(x,y,T)}e^{\int_{\gamma_{xy}}F} dxdy.
\]
We have $\nu_{T}(K) \leq \sum_{i=1}^{l}\nu_{T}(K_{i})$ and
\[
\nu_{T}(K_{i})=\frac{Z_{i}(T)}
{\int_{M\times M}(\sum_{\gamma_{xy}\in G_{\delta,T}}
e^{\int_{\gamma_{xy}}F})dxdy}.
\]
We have
\[
Z_{i}(T)\leq \int_{M\times M}\sum_{\gamma_{xy}\in \Gamma_{i}(x,y,T)}
e^{\int_{\gamma_{xy}}F}e^{l(\gamma_{xy})(\int
  \omega_{i}d\delta_{\gamma _{xy}}-
Q(\omega_{i})-(J(K)-\epsilon))}dxdy.
\]
Set $C:=\sum_{i\leq l}\sup(1, e^{-\delta(-Q(\omega_{i})-(J(K)-\epsilon))})$.
Thus, by taking into account the sign of $-Q(\omega_{i})-(J(K)-\epsilon)$,
\begin{eqnarray*}
Z_{i}(T) &\leq& \int_{M\times M}\sum_{\gamma_{xy}\in \Gamma_{i}(x,y,T)}
e^{\int_{\gamma_{xy}}(F+\omega_{i})}e^{l(\gamma_{xy})\left ( 
-Q(\omega_{i})-(J(K)-\epsilon)\right )}dxdy\\
&\leq& C e^{T\left ( 
-Q(\omega_{i})-(J(K)-\epsilon)\right )}
\int_{M\times M}\sum_{\gamma_{xy}\in \Gamma_{i}(x,y,T)}
e^{\int_{\gamma_{xy}}(F+\omega_{i})}dxdy.
\end{eqnarray*}
For $T$ sufficiently large, it follows from Theorem $9$ $(1)$,
\begin{eqnarray*}
\nu_{T}(K)& \leq& \sum_{i=1}^{l} \frac{Z_{i}(T)}{\int 
\sum_{\gamma_{xy}\in G_{\delta,T}} e^{\int_{\gamma_{xy}}F}dxdy}\\
&\leq&
C\sum_{i=1}^{l}e^{T(P(F+\omega_{i})+\epsilon)}e^{-T(P(F)-\epsilon)}
e^{T(-Q(\omega_{i})-(J(K)-\epsilon)}\\
&=&Cl e^{T(-J(K)+3\epsilon)}.
\end{eqnarray*}
Take the logarithme, divide by $T$ and the $\limsup$,
\[
\limsup_{T\rightarrow \infty}\frac{1}{T}\log \nu_{T}(K) \leq -J(K)+3\epsilon.
\]
$\epsilon$ being arbitrary, this proves Theorem $1$ $(1)$.
\end{proof}
\subsection{Proof of Corollary 1} 
\begin{proof}
It suffices to apply Theorem $1$ to
the closed set $K=\mathcal{P}(SM)\backslash V$. We have $J(K)=J(m)$
for some $m\in K$ and $\nu_{T}(V)=1-\nu_{T}(K)$. By Remark $1$,
$J(m)>0$ and for $T$ sufficiently large,
\[
1\geq \nu_{T}(V) \geq 1-e^{TJ(m)}.
\]
\end{proof}
\subsection{Proof of Theorem 2 ($2$)}
\begin{proof}
Recall that $\mathcal{P}_{g,\alpha}(SM)=\{m\in
\mathcal{P}(SM): \int gdm=\alpha\}$ and
\begin{equation}
J_{g}(\alpha)= \inf (J(m): m\in
\mathcal{P}_{g,\alpha}(SM))
\end{equation}
\begin{equation}
J_{g}(O_{n}):=\inf_{\alpha \in O_{n}}J_{g}(\alpha).
\end{equation}
Given a continuous function $g:SM\rightarrow \real^{n}$ we set $\beta \cdot
g=\sum_{i=1}^{n}\beta{i}g_{i}$ for $\beta \in \real^{n}$.
Then, by definition of the function $Q$, we have
\begin{equation}
Q(\beta \cdot g)=\sup_{\alpha \in \real^{n}} (\beta \cdot \alpha -
J_{g}(\alpha)), 
\end{equation}
and by duality,
\begin{equation}
J_{g}(\alpha)=\sup_{\beta \in \real^{n}} (\beta \cdot \alpha - Q(\beta
\cdot g)). 
\end{equation}

If $J_{g}(O_{n})= +\infty$ then there is nothing to do. Suppose
then $J(O_{n})< +\infty$.
Let $\varepsilon>0$ and choose $\alpha_{\varepsilon} \in O_{n}$ with
$\mathcal{P}_{g,\alpha_{\varepsilon}}(SM) \ne \emptyset$ such that
\[
J_{g}(O_{n})>J_{g}(\alpha_{\varepsilon})-\varepsilon.
\]
We know from (\cite{Roc} Theorem 23.4 and 23.5) that, given $\alpha$ in
the interior of the affine hull of the domain $D(J_{g})$ of $J_{g}$,
there exists $\beta \in \real^{n}$ such that 
\[
Q(\beta\cdot g)=\beta \cdot\alpha - J_{g}(\alpha).
\]
Let then $\beta_{\varepsilon}\in \real^{n}$ such that
\begin{equation}
Q(\beta_{\varepsilon}\cdot g)=
\beta_{\varepsilon} \cdot \alpha_{\varepsilon}-
J_{g}(\alpha_{\varepsilon}).
\end{equation}
Consider now a small neighborhood of $\alpha_{\varepsilon}$,
\[
O_{n,r}:=\{ \alpha \in \real^{n}:|\alpha_{\varepsilon}-\alpha|\leq r\},
\]
such that $O_{n,r} \subset O_{n}$. Define for any $E \subset \real^{n}$ :
\[
\Gamma_{T}(E):=\{\gamma_{xy}\in G_{\delta,T}:
\delta_{\gamma_{xy}}(g) \in E\}
\]
\[
Z_{T}(E):=
\frac{\int \left(
\sum_{\gamma _{xy}\in \Gamma_{T}(E)}e^{\int_{\gamma
    _{xy}}F}\right) dxdy}{\int \left( \sum_{\gamma _{xy} \in
    G_{\delta,T}}e^{\int_{\gamma _{xy}}F}\right) dxdy}.
\]
We have, $Z_{T}(O_{n})\geq Z_{T}(O_{n,r})$ and
\begin{eqnarray*}
&&\sum_{\gamma _{xy}\in \Gamma_{T}(O_{n,r})}e^{\int_{\gamma _{xy}}F}\\
&=& 
e^{-T\beta_{\varepsilon} \cdot \alpha_{\varepsilon}}
\sum_{\gamma _{xy}\in \Gamma_{T}(O_{n,r})} e^{\int_{\gamma
  _{xy}}F}e^{-T(\beta_{\varepsilon}\cdot 
(\delta_{\gamma_{xy}}(g)-\alpha_{\varepsilon}))}  
  e^{T\beta_{\varepsilon}\cdot \delta_{\gamma_{xy}}(g)} \\
&\geq& 
e^{-T\beta_{\varepsilon} \cdot \alpha_{\varepsilon}}
e^{-r\|\beta_{\varepsilon}\|T}
\sum_{\gamma _{xy}\in \Gamma_{T}(O_{n,r})}e^{\int_{\gamma _{xy}}F}
  e^{T\beta_{\varepsilon}\cdot \delta_{\gamma_{xy}}(g)} \\
&=& 
e^{-T\beta_{\varepsilon} \cdot \alpha_{\varepsilon}}
e^{-r\|\beta_{\varepsilon}\|T}
\sum_{\gamma _{xy}\in \Gamma_{T}(O_{n,r})}e^{\int_{\gamma _{xy}}F}
  e^{l(\gamma _{xy})\beta_{\varepsilon}\cdot 
\delta_{\gamma_{xy}}(g)} 
e^{(T-l(\gamma _{xy}))\beta_{\varepsilon}\cdot 
\delta_{\gamma_{xy}}(g)}. 
\end{eqnarray*}
By the condition, $0\leq T-l(\gamma _{xy}) \leq \delta$, we have
\[
e^{(T-l(\gamma _{xy}))\beta_{\varepsilon}\cdot 
\delta_{\gamma_{xy}}(g)}
\geq e^{-\delta \|\beta_{\varepsilon}\cdot g \|_{\infty}}.
\]
Thus
\[
\sum_{\gamma _{xy}\in \Gamma_{T}(O_{n,r})}e^{\int_{\gamma_{xy}}F}
\geq e^{-T \beta_{\varepsilon} \cdot \alpha_{\varepsilon}}
e^{-r\|\beta_{\varepsilon}\| T}
e^{-\delta \|\beta_{\varepsilon}\cdot g \|_{\infty}}
\sum_{\gamma_{xy}\in \Gamma_{T}(O_{n,r})}
e^{\int_{\gamma_{xy}}(F+\beta_{\varepsilon}\cdot g)}.
\]
Set
\[
Z_{T}^{\varepsilon}(O_{n,r}):=
\frac{\int (\sum_{\gamma_{xy}\in\Gamma_{T}(O_{n,r})}
e^{\int_{\gamma_{xy}}(F+\beta_{\varepsilon}\cdot g)})dxdy} 
{\int ( \sum_{\gamma_{xy} \in G_{\delta,T}}
e^{\int_{\gamma_{xy}}(F+\beta_{\varepsilon}\cdot g)})dxdy}  
\]
and
\[
Z_{T}(\beta_{\varepsilon}\cdot g):=\frac{1}{T}\log 
\frac{\int (\sum_{\gamma_{xy}\in G_{\delta,T}}
e^{\int_{\gamma_{xy}}(F+\beta_{\varepsilon}\cdot g)})dxdy} 
{\int \sum_{\gamma_{xy} \in G_{\delta,T}}
e^{\int_{\gamma_{xy}}F}dxdy}.
\]
Therefore,
\begin{equation}
\frac{1}{T}\log Z_{T}(O_{n,r}) \geq 
-r\|\beta_{\varepsilon}\|-
\frac{\delta \|\beta_{\varepsilon} \cdot g \|_{\infty}}{T}+ 
(Z_{T}(\beta_{\varepsilon}\cdot g)-\beta_{\varepsilon}\cdot
\alpha_{\varepsilon})+ 
\frac{1}{T}\log Z_{T}^{\varepsilon}(O_{n,r}), 
\end{equation}
From Theorem $9$ $(1)$ and by definition of $Q=Q_{F}$ (see $(2)$) we get,
\[
\lim_{T\rightarrow \infty}
Z_{T}(\beta_{\varepsilon}\cdot g)=
P(F+\beta_{\varepsilon}\cdot g)-P(F)=Q(\beta_{\varepsilon}\cdot g).
\]
Thus
\[
\liminf_{T\rightarrow \infty}\frac{1}{T}\log Z_{T}(O_{n})\geq
-r\|\beta_{\varepsilon}\|+
 (Q(\beta_{\varepsilon}\cdot g)-\beta_{\varepsilon}\cdot
\alpha_{\varepsilon})+\lim_{T\rightarrow \infty}\frac{1}{T}\log
 Z_{T}^{\varepsilon}(O_{n,r}).
\]
We will show that
\begin{equation}
\lim_{T\rightarrow \infty} Z_{T}^{\varepsilon}(O_{n,r})=1.
\end{equation}
Let us see how to finish the proof using $(14)$ :
\begin{eqnarray*}
\liminf_{T\rightarrow \infty} \frac{1}{T}\log Z_{T}(O_{n})&\geq&
-r\|\beta_{\varepsilon}\|+Q(\beta_{\varepsilon}\cdot g)-\beta_{\varepsilon}
\cdot \alpha_{\varepsilon}\\ 
&=&-r\|\beta_{\varepsilon}\|-J_{g}(\alpha_{\varepsilon}) \\
&\geq& -r\|\beta_{\varepsilon}\|-J_{g}(O_{n}) -\varepsilon, 
\end{eqnarray*}
for any $\varepsilon >0$. Since $r>0$ was arbitray choosen, we let
$r\rightarrow 0$ and $\varepsilon \rightarrow 0$ respectively and we get
$\liminf_{T\rightarrow \infty} \frac{1}{T}\log Z_{T}(O_{n})\geq
-J_{g}(O_{n})$ which completes the proof Theorem $2$ ($2$). 

It remains to show ($14$). Let $K_{n,r}$ be the complement set of
$O_{n,r}$ in the image 
$g*(\mathcal{P}(SM))$ of $\mathcal{P}(SM)$ under the continuous map
$g* : m\rightarrow g\cdot m$. We have
$Z_{T}^{\varepsilon}(O_{n,r})+Z_{T}^{\varepsilon}(K_{n,r})=1$.
The goal is to show that $Z_{T}^{\varepsilon}(K_{n,r})$ decrease exponentially
fast to zero as $T\rightarrow \infty$ using Theorem $1$.

Consider $J^{\varepsilon}:=J_{F+\beta_{\varepsilon}\cdot g}$ which is
the functional $J$ corresponding to
$Q^{\varepsilon}:=Q_{F+\beta_{\varepsilon}\cdot g}$. We have
\[
Q^{\varepsilon}(\omega) =P(F+\beta_{\varepsilon}\cdot
g+\omega)-P(F+\beta_{\varepsilon}\cdot g), 
\]
and by $(3)$,
\[
J^{\varepsilon}(m)=\sup_{\omega}(\int \omega dm -Q^{\varepsilon}(\omega)).
\]
From this we deduce easily that
\[
J^{\varepsilon}(m)=J(m)+Q(\beta_{\varepsilon}\cdot g)-\int
\beta_{\varepsilon}\cdot g dm, 
\]
and
\[
\inf_{m(g)=\alpha}J^{\varepsilon}(m)=\inf_{m(g)=\alpha}J(m)
+Q(\beta_{\varepsilon}\cdot g)-\beta_{\varepsilon}\cdot \alpha.
\]
The set $K_{n,r}$ is compact in $\real^{n}$ and by Theorem $1$ ($1$),
\[
\limsup_{T\rightarrow \infty} \frac{1}{T}\log
Z_{T}^{\varepsilon}(K_{n,r}) \leq - J^{\epsilon}(K),
\]
with $K:=(g*)^{-1}(K_{n,r})$ which is a closed subset of
$\mathcal{P}(SM)$.
We have,
\[
J^{\varepsilon}(K)=\inf_{\alpha \in
  K_{n,r}}(J_{g}(\alpha)+Q(\beta_{\varepsilon}\cdot
g)-\beta_{\varepsilon}\cdot \alpha). 
\]
If $J^{\varepsilon}(K)=+\infty$ there is nothing to do and the
result follows. The key point is to prove that $J^{\varepsilon}(K)>0$.
Set
\[
J_{g}^{\varepsilon}(\alpha):=
J_{g}(\alpha)+Q(\beta_{\varepsilon}\cdot g)-\beta_{\varepsilon}   
\cdot \alpha.
\]
The functional $J^{\varepsilon}$ is non negative (since
$Q^{\varepsilon}(0)=0$), lower 
semicontinuous and then it achieves its minimum on compact sets. 
We have $J_{g}^{\varepsilon}(\alpha)\geq 0$ and 
$J_{g}^{\varepsilon}(\alpha_{\varepsilon})=0$ (see $(12)$). Recall that,
if $J_{g}^{\varepsilon}(\alpha)=0$ for some $\alpha \in K_{n,r}$, then
there will 
correspond to $\alpha$ an equilibrium state $m_{\alpha}\in K$ for the
potential $F+\beta_{\varepsilon}\cdot g$ such that $m_{\alpha}(g)=\alpha$. 
The vector $\alpha_{\varepsilon}$ is the unique point realizing the
minimum, i.e the unique solution for the
equation $J_{g}^{\varepsilon}(\alpha)=0$.
Indeed, two different solutions will produce two distinct
equilibrium states for the potential $F+\beta_{\varepsilon}\cdot
g$ which contradicts our standing assumption of Theorem $1$.
Since $\alpha_{\varepsilon} \in O_{n,r}$, then
$J_{g}^{\varepsilon}(\alpha)>0$ for $\alpha \in K_{n,r}$. On the other hand
the set $K_{n,r}$ being compact, by the lower semicontinuity of
$J^{\varepsilon}$ we have $J^{\varepsilon}(K)=\inf_{\alpha \in
  K_{n,r}} J_{g}^{\varepsilon}(\alpha)>0$. Thus we have proved that
\[
\limsup_{T\rightarrow \infty} \frac{1}{T}\log
Z_{T}^{\varepsilon}(K_{n,r}) \leq -J^{\varepsilon}(K)<0
\]
from which ($14$) follows immediately.
\end{proof}

\subsection{Proof of Theorem $1$ ($2$)}
\begin{proof}
Let $O \subset \mathcal{P}(SM)$ be an open set and $\epsilon
>0$. Choose $m_{\epsilon} \in O$ such that 
\[
J(m_{\epsilon}) \leq J(O) +\epsilon.
\]
We endow the space $\mathcal{P}(SM)$ with
a compatible topology generated by the distance given by 
\[
d(m,m'):=\sum_{k=1}^{\infty}2^{-k}|m(g_{k})-m'(g_{k})|,
\] where the functions $g_{k}$ were defined in section $2.2$.
Following \cite{Ki} we define,
\[
d_{n}(m,m'):=\sum_{k=1}^{n}2^{-k}|m(g_{k})-m'(g_{k})|.
\]
Set $2r=\inf \{d(m, m_{\epsilon}): m\in \mathcal{P}(SM)\backslash O\}$.
We have $r>0$, since $\mathcal{P}(SM)\backslash O$ is a compact
subset of $\mathcal{P}(SM)$.
Since for all $k$, $\|g_{k}\|=1$, we have $0\leq d(m,
m')-d_{n}(m, m')\leq 2^{-(n-1)}$. Thus, for $n$ suffuciently
large,
\[
O_{\epsilon, r}:=\{m\in \mathcal{P}(SM):d_{n}(m, m_{\epsilon})<r\}
\subset O. 
\]
For each $\alpha =(\alpha_{1}, \ldots, \alpha_{n})\in \real^{n}$
write $\|\alpha\|_{n}=\sum_{k=1}^{n}2^{-k}|\alpha_{k}|$. Set
$\alpha_{\epsilon}:=(\int
g_{1}dm_{\epsilon},\ldots, \int g_{n}dm_{\epsilon})=m_{\epsilon}(g^{(n)})$ and
\[
O_{n,r}:=\{\alpha \in \real^{n}: \|\alpha_{\epsilon}-\alpha \|_{n}<r\}.
\]
Then, $g^{(n)}(O_{\epsilon, r})=O_{n,r}\cap g^{(n)}(\mathcal{P}(SM))$.
From Theorem $2$ ($2$) we get, 
\begin{eqnarray*}
&& \liminf_{T\rightarrow +\infty }\frac{1}{T}\log \frac{\int \left(
\sum_{\gamma _{xy}\in G_{\delta,T}:
  \delta_{\gamma_{xy}}\in O} 
e^{\int_{\gamma _{xy}}F}\right) dxdy}{\int \left( \sum_{\gamma
_{xy} \in G_{\delta,T}}e^{\int_{\gamma
_{xy}}F}\right) dxdy}\\
&\geq& 
\liminf_{T\rightarrow +\infty }\frac{1}{T}\log \frac{\int \left(
\sum_{\gamma _{xy}\in G_{\delta,T}:  \delta_{\gamma_{xy}}\in
  O_{\epsilon, r} 
}e^{\int_{\gamma _{xy}}F}\right) dxdy}{\int \left( \sum_{\gamma
_{xy} \in G_{\delta,T}}e^{\int_{\gamma
_{xy}}F}\right) dxdy}\\
&=& \liminf_{T\rightarrow +\infty }\frac{1}{T}\log \frac{\int \left(
\sum_{\gamma _{xy}\in G_{\delta,T}:
\delta_{\gamma_{xy}}(g^{(n)})\in O_{n, r}
}e^{\int_{\gamma _{xy}}F}\right) dxdy}{\int \left( \sum_{\gamma
_{xy} \in G_{\delta,T}}e^{\int_{\gamma_{xy}}F}\right) dxdy}\\
&\geq& -J_{g}(O_{n, r})\\
&\geq& -J_{g}(\alpha_{\epsilon})
\geq -J(m_{\epsilon}) \geq -J(O) -\epsilon,
\end{eqnarray*}
for any $\epsilon >0$. This completes the proof of the main Theorem 1.
\end{proof}

\subsection{Proof of Theorem $3$}
\begin{proof} We have to show that the weak limits of
\[
m_{\delta,T}:=\frac{\int \left ( \sum_{\gamma _{xy}\in G_{\delta,T}}
e^{\int_{\gamma_{xy}}F}\delta_{\gamma_{xy}}\right 
)dxdy}{\int \left ( \sum_{\gamma _{xy}\in
G_{\delta,T}}e^{\int_{\gamma_{xy}}F}\right )dxdy},
\]are contained in $\mathcal{P}_{e}(F)$.

Let $V\subset \mathcal{P}(SM)$ be a convex open 
neighborhood of $\mathcal{P}_{e}(F)$ and $\epsilon >0$.
We consider a finite open cover $(B_{i}(\epsilon))_{i\leq N}$ of
$\mathcal{P}_{e}(F)$ by balls of diameter $\epsilon$ all contained in $V$. 
Define the measures on $SM$,
\[
m_{T,V}:=\frac{\int \left ( \sum_{\gamma _{xy}\in G_{\delta,T}:
\delta_{\gamma_{xy}} \in V}
e^{\int_{\gamma_{xy}}F}\delta_{\gamma_{xy}}\right 
)dxdy}{\int \left ( \sum_{\gamma _{xy}\in
G_{\delta,T}}e^{\int_{\gamma_{xy}}F}\right )dxdy}.
\]
Decompose the set $U:=\cup_{i=1}^{N}B_{i}(\epsilon)$ in a disjoint
union as follows,
\[
U=\cup_{j=1}^{N'}U_{j}^{\epsilon},
\]where the sets $U_{j}^{\epsilon}$ are disjoints and contained in one
of the balls $(B_{i}(\epsilon))_{i\leq N}$. We have
\[
\mathcal{P}_{e}(F) \subset U \subset V.
\] 
We fix in each $U_{j}^{\epsilon}$ a probability measure $m_{j}$,
$j\leq N'$, and let $m_{0}$ be a probability  measure distinct from the above
ones (for example take $m_{0} \in V\backslash U$).

Set as usual,
\begin{equation}
\nu_{T}(E):=\frac{\int \left ( \sum_{\gamma _{xy}\in G_{\delta,T}:
\delta_{\gamma_{xy}} \in E}
e^{\int_{\gamma_{xy}}F}\right 
)dxdy}{\int \left ( \sum_{\gamma _{xy}\in
G_{\delta,T}}e^{\int_{\gamma_{xy}}F}\right )dxdy},
\end{equation}
and define,
\begin{equation}
\beta_{T}=
\sum_{j=1}^{N'}\nu_{T}(U_{j}^{\epsilon})m_{j}+
(1-\nu_{T}(U))m_{0}. 
\end{equation}
We have $\sum_{j=1}^{N'}\nu_{T}(U_{j}^{\epsilon})=\nu_{T}(U)$.
The probability measure $\beta_{T}$ lies in $V$ since it is a convex
combination of elements in the convex set $V$. We have then
$d(m_{\delta,T}, V)\leq d(m_{\delta,T},\beta_{T})$. We
will show that
\[
d(m_{\delta,T},\beta_{T})\leq \epsilon \nu_{T}(U)+\frac{3}{2}\nu_{T}(U^{c}),
\]
where $U^{c}=\mathcal{P}(SM)\backslash U$ which is closed.

By definition of $m_{\delta,T}$ and $m_{T,V}$ and
the fact that $U\subset V$,
\[
\sum_{k\geq 1}2^{-k}|m_{\delta,T}(g_{k})-m_{T,V}(g_{k})|\leq
\frac{1}{2} \nu_{T}(U^{c}).
\]
It remains to show that $d(m_{T,V},\beta_{T}) \leq \epsilon
\nu_{T}(U)+\nu_{T}(U^{c})$. 
We have for all $k\geq 1$,
\[
|m_{T,V}(g_{k})-\beta_{T}(g_{k})|\leq A+B+C
\]where,
\[
A=
\frac{\sum_{j=1}^{N'}\int \left ( \sum_{\gamma _{xy}\in G_{\delta,T}:
\delta_{\gamma _{xy}}\in U_{j}^{\epsilon}}
e^{\int_{\gamma_{xy}}F}|\delta_{\gamma_{xy}}(g_{k})-m_{j}(g_{k})| 
\right )
dxdy}{\int \left ( \sum_{\gamma _{xy}\in
G_{\delta,T}}e^{\int_{\gamma_{xy}}F}\right )dxdy},
\]
\[
B=\frac{\int(\sum_{\gamma _{xy}\in G_{\delta,T} :
\delta_{\gamma _{xy}}\in V\backslash U}
e^{\int_{\gamma_{xy}}F}\delta_{\gamma _{xy}}(g_{k}))dxdy}
{\int \left ( \sum_{\gamma _{xy}\in
G_{\delta,T}}e^{\int_{\gamma_{xy}}F}\right )dxdy},
\]
\[
C=|(1-\nu_{T}(U))m_{0}(g_{k})|.
\]
Thus, since we have for all $k\geq 1$, $\|g_{k}\|=1$,  by definition
of $\nu_{T}$ $(15)$ we get,
\begin{eqnarray*}
&&\sum_{k\geq 1}2^{-k}|m_{T,V}(g_{k})-\beta_{T}(g_{k})|\\
&\leq& \epsilon
  \sum_{j=1}^{N'}\nu_{T}(U_{j}^{\epsilon})+\frac{1}{2}\nu_{T}(U^{c})+
\frac{1}{2}(1-\nu_{T}(U))\\
&=& \epsilon \nu_{T}(U)+\nu_{T}(U^{c}).
\end{eqnarray*}
Finally we have obtained that
\[
d(m_{\delta,T},\beta_{T})\leq \epsilon
\nu_{T}(U)+\frac{3}{2}\nu_{T}(U^{c}). 
\]
This implies the desired inequality,
\[
d(m_{\delta, T},V)\leq \epsilon
\nu_{T}(U)+\frac{3}{2}\nu_{T}(U^{c}). 
\]
By Corollary $1$, since $U^{c}$ is closed, we know that $\lim_{T\rightarrow
  \infty}\nu_{T}(U)=1$. Thus, $\limsup_{T\rightarrow
  \infty}d(m_{\delta,T},V)\leq \epsilon$, for all $\epsilon >0$. We
  conclude that $\limsup_{T\rightarrow
  \infty}d(m_{\delta,T},V)=0$. The neighborhood $V$ of
  $\mathcal{P}_{e}(F)$ being arbitrary, this implies that all limit
  measures of $m_{\delta,T}$ are contained in $\mathcal{P}_{e}(F)$. In
  particular, if $\mathcal{P}_{e}(F)$ is reduced to one measure $\mu$, this
  shows that $m_{T}$ converges to $\mu$.
\end{proof}
\subsection{Proof of Theorem $5$}

\subsubsection{Proof of Part $1$}
\begin{proof}
Set for any $E\subset \mathcal{P}(SM)$,
\[
\nu_{T}(E)=\frac{\int_{M\times M}
(\sum_{\gamma_{xy}\in G_{T}:\delta_{\gamma_{xy}}\in E}
e^{\int_{\gamma_{xy}}F})dxdy} 
{\int_{M\times M}(\sum_{\gamma_{xy}\in G_{T}}
e^{\int_{\gamma_{xy}}F})dxdy}.
\]
Let $\epsilon >0$. There exists a finite number of continuous
functions $\omega_{1}, \cdots, \omega_{l}$ such that $K\subset
\cup_{i=1}^{l}K_{i}$, where 
\[
K_{i}=\{m\in \mathcal{P}(SM): \int \omega_{i}dm-Q(\omega_{i})>J(K)-\epsilon\}.
\]
We can suppose that all the $\omega_{i}$'s are non negative since adding a
constant $c>0$ we have,
\[
\int (\omega_{i}+c)dm-Q(\omega_{i}+c)=\int \omega_{i}dm-Q(\omega_{i}).
\]
Put
\[
\Gamma_{i}(x,y,T):=\{\gamma_{xy}\in
G_{T}:\delta_{\gamma_{xy}}\in K_{i}\},  
\]and
\[
Z_{i}(T):=\int \sum_{\gamma_{xy}\in
  \Gamma_{i}(x,y,T)}e^{\int_{\gamma_{xy}}F}dxdy. 
\]
From the definition of $\Gamma_{i}(x,y,T)$ we get,
\[
Z_{i}(T)\leq \int \sum_{\gamma_{xy}\in \Gamma_{i}(x,y,T)}
e^{\int_{\gamma_{xy}}F}e^{l(\gamma_{xy})(\int
  \omega_{i}d\delta_{\gamma_{xy}}-
Q(\omega_{i})-(J(K)-\epsilon))}dxdy.
\]
We decompose the interval $[0, T]$ into subintervals $[T-(j+1)\delta,
  T-j\delta]$ and set 
\[
\Gamma_{ij}(x,y,T):=\{\gamma_{xy}:T-(j+1)\delta<l(\gamma_{xy}) \leq
T-j\delta, \  \delta_{\gamma_{xy}}\in K_{i}\}.  
\]
Since the functions $\omega_{i}$ were supposed non negative, we have
$Q(\omega_{i})\geq 0$ and then $Q(\omega_{i})+J(K) \geq 0$. Thus, from
Theorem $9$ $(1)$ and for $T$ sufficiently large,  
\begin{eqnarray*}
Z_{i}(T)&\leq& \int \sum_{\gamma_{xy}\in \Gamma_{i}(x,y,T)}
e^{\int_{\gamma_{xy}}(F+\omega_{i})}e^{l(\gamma_{xy}) \left ( 
-Q(\omega_{i})-(J(K)-\epsilon)\right )}dxdy\\
&\leq& \sum_{j}e^{(T-(j+1)\delta)\left (-Q(\omega_{i})-(J(K)-\epsilon)\right )}
\int \sum_{\gamma_{xy}\in \Gamma_{ij}(x,y,T)}
e^{\int_{\gamma_{xy}}(F+\omega_{i})}dxdy\\
&\leq& \sum_{j} e^{(T-(j+1)\delta)\left
(-Q(\omega_{i})-(J(K)-\epsilon)\right )} 
e^{(T-j\delta)(P(F+\omega_{i})+\epsilon)}\\
&=& \sum_{j}e^{(T-j\delta)(P(F)-J(K)+2\epsilon)}
e^{-\delta(-Q(\omega_{i})-(J(K)-\epsilon))}.
\end{eqnarray*}
We assumed that $J(K)< P(F)$, then
$C:=\sum_{j}e^{-j\delta (P(F)-J(K)+2\epsilon)}<\infty$. Setting
$\lambda_{i}:=e^{-\delta(-Q(\omega_{i})-(J(K)-\epsilon))}$, we get
for $T$ sufficiently large and Theorem $9$ $(2)$ ($P(F)> 0$),
\begin{eqnarray*}
\nu_{T}(K) &\leq& 
 \frac{\sum_{i=1}^{l}Z_{i}(x,y,T)}{\int  
\sum_{\gamma_{xy}\in G_{T}} e^{\int_{\gamma_{xy}}F}dxdy}\\
&\leq&
\sum_{i=1}^{l}\sum_{j}
e^{T(-J(K)+3\epsilon)}\lambda_{i}e^{-j\delta (P(F)-J(K)+2\epsilon)}\\
&=&e^{T(-J(K)+3\epsilon)}C\sum_{i=1}^{l}\lambda_{i} . 
\end{eqnarray*}
Take the logarithme, divide by $T$ and take the $\limsup$,
\[
\limsup_{T\rightarrow \infty}\frac{1}{T}\log \nu_{T}(K) \leq -J(K)+3\epsilon.
\]
$\epsilon$ being arbitrary, this proves Theorem $5$ $(1)$.
\end{proof}

\subsubsection{Proof of Part $2$}
\begin{proof} As for the proof of part $(2)$ of Theorem $1$, the proof
will be a consequence of the contraction principle,
\[
\liminf_{T\rightarrow \infty}
\frac{\int \left(
\sum_{\gamma _{xy}\in G_{T}:
\delta_{\gamma _{xy}}(g)\in O_{n} 
}e^{\int_{\gamma _{xy}}F}\right) dxdy}{\int \left( \sum_{\gamma
_{xy} \in G_{T}}e^{\int_{\gamma
_{xy}}F}\right) dxdy} \geq -J_{g}(O_{n}).
\]
We follow the lines of the proof
of Theorem $2$ $(2)$ with the same notations. 
Let $\delta >0$. We have,
\[
\frac{\int \left(
\sum_{\gamma _{xy}\in G_{T}:
\delta_{\gamma _{xy}}(g)\in O_{n} 
}e^{\int_{\gamma _{xy}}F}\right) dxdy}{\int \left( \sum_{\gamma
_{xy} \in G_{T}}e^{\int_{\gamma
_{xy}}F}\right) dxdy} \geq
 \frac{\int \left(
\sum_{\gamma _{xy}\in \Gamma_{T}(O_{n})
}e^{\int_{\gamma _{xy}}F}\right) dxdy}{\int \left( \sum_{\gamma
_{xy} \in G_{T}}e^{\int_{\gamma
_{xy}}F}\right) dxdy},
\]
where,
\[
\Gamma_{T}(O_{n}):=
\{\gamma _{xy}\in G_{\delta,T}:\delta_{\gamma _{xy}}(g)\in O_{n}\}.
\]
Set
\[
Z_{T}(O_{n}):=
\frac{\int \left(
\sum_{\gamma _{xy}\in \Gamma_{T}(O_{n})}e^{\int_{\gamma
    _{xy}}F}\right) dxdy}{\int \left( \sum_{\gamma _{xy} \in
    G_{T}}e^{\int_{\gamma _{xy}}F}\right) dxdy}.
\]
Then $ Z_{T}(O_{n}) \geq Z_{T}(O_{n,r})$ and (see $(13)$),
\begin{eqnarray*}
&&\frac{1}{T}\log Z_{T}(O_{n})\\
&\geq &\frac{1}{T}\log Z_{T}(O_{n,r})\\
&\geq& 
-r\|\beta_{\varepsilon}\|-
\frac{\delta \|\beta_{\varepsilon} \cdot g \|_{\infty}}{T}+ 
(Z_{T}(\beta_{\varepsilon}\cdot g)-\beta_{\varepsilon}\cdot
\alpha_{\varepsilon})+ 
\frac{1}{T}\log Z_{T}^{\varepsilon}(O_{n,r}), 
\end{eqnarray*}
where we have set,
\[
Z_{T}^{\varepsilon}(O_{n,r}):=
\frac{\int (\sum_{\gamma_{xy}\in\Gamma_{T}(O_{n,r})}
e^{\int_{\gamma_{xy}}(F+\beta_{\varepsilon}\cdot g)})dxdy} 
{\int ( \sum_{\gamma_{xy} \in G_{\delta,T}}
e^{\int_{\gamma_{xy}}(F+\beta_{\varepsilon}\cdot g)})dxdy},  
\]
and
\[
Z_{T}(\beta_{\varepsilon}\cdot g):=\frac{1}{T}\log 
\frac{\int (\sum_{\gamma_{xy}\in G_{\delta,T}}
e^{\int_{\gamma_{xy}}(F+\beta_{\varepsilon}\cdot g)})dxdy} 
{\int \sum_{\gamma_{xy} \in G_{T}}
e^{\int_{\gamma_{xy}}F}dxdy}.
\]
From Theorem $9$ $(1)$ and $(2)$ respectively,
\[
\lim_{T\rightarrow
  \infty}Z_{T}(\beta_{\varepsilon}\cdot g)=
P(F+\beta_{\varepsilon}\cdot g)-P(F)=Q(\beta_{\varepsilon}\cdot g).
\]
We proved in $(14)$ that $\lim_{T\rightarrow
  \infty}Z_{T}^{\varepsilon}(O_{n,r}) =1$. Thus,
\begin{eqnarray*}
&&\liminf_{T\rightarrow \infty}\frac{1}{T}\log 
\frac{\int \left(
\sum_{\gamma _{xy}\in G_{T}:
\delta_{\gamma _{xy}}(g)\in O_{n}} 
e^{\int_{\gamma _{xy}}F}\right) dxdy}{\int \left( \sum_{\gamma
_{xy} \in G_{T}}e^{\int_{\gamma
_{xy}}F}\right) dxdy} \\
&\geq&\liminf_{T\rightarrow \infty}\frac{1}{T}\log 
\frac{\int \left(
\sum_{\gamma _{xy}\in G_{\delta,T}:
\delta_{\gamma _{xy}}(g)\in O_{n}} 
e^{\int_{\gamma _{xy}}F}\right) dxdy}{\int \left( \sum_{\gamma
_{xy} \in G_{T}}e^{\int_{\gamma
_{xy}}F}\right) dxdy}\\
&\geq&
-r\|\beta_{\varepsilon}\|+
 (Q(\beta_{\varepsilon}\cdot g)-\beta_{\varepsilon}\cdot
\alpha_{\varepsilon})+\lim_{T\rightarrow \infty}\frac{1}{T}\log
 Z_{T}^{\varepsilon}(O_{n,r})\\
&=& -r\|\beta_{\varepsilon}\|-I_{g}(\alpha_{\varepsilon}) \\
&\geq& -r\|\beta_{\varepsilon}\|-J_{g}(O_{n}) -\varepsilon, 
\end{eqnarray*}
for any $\varepsilon >0$. Since $r>0$ was arbitray choosen, we let
$r\rightarrow 0$ and $\varepsilon \rightarrow 0$ respectively, 
this completes the proof of the contraction principle.
\end{proof}

\subsection{Proof of Theorem $4$}
\begin{proof} We have to prove that the weak limits of
\[
m_{T}:=\frac{\int \left ( \sum_{\gamma _{xy}\in
    G_{T}}e^{\int_{\gamma_{xy}}F}\delta_{\gamma _{xy}}\right )dxdy}
{\int \left ( \sum_{\gamma _{xy}\in
G_{T}}e^{\int_{\gamma_{xy}}F}\right )dxdy}
\]are in $\mathcal{P}_{e}(F)$. For this and in order to follow the
proof of Theorem $3$, we must show that we are able to
apply  Theorem $5$ $(1)$ to any open and convex
neighborhood $V$ of $\mathcal{P}_{e}(F)$, $V\subset \mathcal{P}(SM)$.
There are minor changes due to the
conditions $P(F)> 0$ and $J(K)<P(F)$ in Theorem
$5$ $(1)$.
Let $\nu \in \mathcal{P}_{e}(F)$ and $\mu$ a probability measure which
is not invariant, i.e a non invariant element in $\mathcal{P}(SM)
\backslash \mathcal{P}_{e}(F)$.
We have $J(\nu)=0$ and $J(\mu)>0$ (by
Remark $1$). Consider a convex sum of these two measures, $m=\alpha
\mu +\beta \nu$, where $\alpha + \beta =1$ ($\alpha \ne 0$ and $\beta
\ne 0$). Observe
that $m\notin \mathcal{P}_{e}(F)$. Indeed, assume that $m$ lies in
$\mathcal{P}_{e}(F)$. In particular it is then invariant. But then,
since $\nu$ is invariant by assumption, for
all $t$ we will have
\[
\alpha \mu+ \beta \nu=\alpha \mu\circ \phi_{t} +\beta \nu \circ \phi_{t}=
\alpha \mu\circ \phi_{t} +\beta \nu.
\]  
Thus $\mu =\mu\circ \phi_{t}$ for all $t$, which is a contradiction.

From the convexity of $J$ we get, $J(m)\leq \alpha J(\mu)$.  
Let $V\subset \mathcal{P}(SM)$ be an open and convex
small neighborhood of $\mathcal{P}_{e}(F)$ such that $m\in
V^{c}:=K$ (note that in particular, for all other contained small
neighborhood $V$ we will have $m \in V^{c}$). Therefore,
\begin{equation}
\inf_{k\in K} J(k)\leq J(m)\leq \alpha J(\mu).
\end{equation}
Now, since we have assumed that $P(F)>0$
then for a sufficiently small $\alpha >0$ we get from $(17)$,
\[
J(K):=\inf_{k\in K} J(k)\leq P(F).
\]
Thus $P(F)-J(K)\geq 0$. We can then apply Theorem $5$ $(1)$ to
the closed sets $K=V^{c}$ and conclude with the proof of Theorem $3$.
\end{proof}
\subsection{Proof of Theorem $6$}
\subsubsection{Proof of Part $(1)$}
\begin{proof}
It is essentially the proof of Theorem $1$ $(1)$ with the following
modifications since it is in part based on Theorem $8$ and the
following lemma (see \cite{Ma} Lemma 4.3, \cite{Pate} Lemma 3.33 p$68$).
\begin{lemma}
Let $(X, \mathcal{A}, \mu)$ be a probability space, and $f_{n} :
X\rightarrow (0, +\infty)$ a sequence of integrable functions. Then
for $\mu$ a.e $x\in X$
\[
\limsup_{n\rightarrow \infty} \frac{1}{n}\log f_{n}(x)
\leq \limsup_{n\rightarrow \infty} \frac{1}{n}\log \int_{X} f_{n}d\mu.
\]
\end{lemma}
Set $F=0$ and proceed as in the proof of Theorem $1$ $(1)$ with the
same notations. We have $Q(\omega)=P(\omega)-h_{top}$ and,
\begin{eqnarray*}
Z_{i}(x,y,T) &\leq& \sum_{\gamma_{xy}\in \Gamma_{i}(x,y,T)}
e^{\int_{\gamma_{xy}}\omega_{i}}e^{l(\gamma_{xy})\left ( 
-Q(\omega_{i})-(J(K)-\epsilon)\right )}\\
&\leq& C e^{T\left ( 
-Q(\omega_{i})-(J(K)-\epsilon)\right )}\sum_{\gamma_{xy}\in \Gamma_{i}(x,y,T)}
e^{\int_{\gamma_{xy}}\omega_{i}}.
\end{eqnarray*}
We have
\[
\nu_{T}(K) \leq \sum_{i=1}^{l} \frac{Z_{i}(x,y,T)}{\# G_{\delta,T}}.
\]
For all $(x,y)$ and $T$ sufficiently large (depending on $(x,y)$), it
follows from Theorem $8$ $(2)$ that, 
\[
\#G_{\delta,T} \geq e^{T(h_{top}-\epsilon)}.
\]
On the other hand it follows from Lemma $2$ above (which can be
applied to continuous time) and Theorem $9$ $(1)$,
\begin{eqnarray*}
&&\limsup_{T\rightarrow \infty} \frac{1}{T}\log
\sum_{\gamma_{xy}\in \Gamma_{i}(x,y,T)}
e^{\int_{\gamma_{xy}}\omega_{i}}\\
&\leq& \limsup_{T\rightarrow \infty} \frac{1}{T}\log
\sum_{\gamma_{xy}\in  G_{\delta,T} }
e^{\int_{\gamma_{xy}}\omega_{i}}\\
&\leq& \limsup_{T\rightarrow \infty} \frac{1}{T}\log
\int \left (\sum_{\gamma_{xy}\in G_{\delta,T}}
e^{\int_{\gamma_{xy}}\omega_{i}}\right )dxdy\\
&\leq& P(\omega_{i}),
\end{eqnarray*}
for $(x,y)$ in a subset $B_{i}$ of $M\times M$ of full Lebesgue
measure and $i\leq l$.
The set $B_{i}$ can be taken independent from $i$ for evident
reasons, and then we will say that the above inequalities hold for a.e
$(x,y)$ in a set $B$. However, note that $B$ depends on the set of
functions $\{\omega_{i}, i\leq l\}$ which means that $B$ depends on $K$.
The proof can be now achieved similarily : let $(x,y)\in B$ fixed and $T$
sufficiently large (depending on $(x,y)$),
\begin{eqnarray*}
\nu_{T}(K)& \leq& \sum_{i=1}^{l} \frac{Z_{i}(x,y,T)}{ 
\# G_{\delta,T}}\\
&\leq&
C\sum_{i=1}^{l}e^{T(P(\omega_{i})+\epsilon)}e^{-T(h_{top}-\epsilon)}
e^{T(-Q(\omega_{i})-(J(K)-\epsilon)}\\
&=&Cl e^{T(-J(K)+3\epsilon)}.
\end{eqnarray*}
Thus $\limsup_{T\rightarrow \infty} \frac{1}{T}\log \nu_{T}(K)
\leq -J(K)$.
\end{proof}
\subsubsection{Proof of Part $(2)$}
\begin{proof} Set
\[
\nu_{T}(E)=\frac{\#\{\gamma_{xy}\in G_{\delta,T}:\delta_{\gamma_{xy}}\in E\}} 
{\# G_{\delta,T}},
\]
and
\[
\mu_{T,V}(x,y):=\frac{\sum_{\gamma_{xy}\in
  G_{\delta,T}:\delta_{\gamma_{xy}}\in V}\delta_{\gamma_{xy}}}  
{\# G_{\delta,T}}.
\]
Thus, proceeding as in the proof of Theorem $3$, we get for any open and
convex neighborhood $V$ of $\mathcal{P}_{e}(0)$ and all $(x,y)$,
\begin{equation}
d(\mu_{T}(x,y),V)\leq \epsilon \nu_{T}(U)+\frac{3}{2}\nu_{T}(U^{c}).
\end{equation}
Let $(V_{i})_{i\geq 1}$ be a decreasing sequence of sets of the type $V$
such that $\cap_{i\geq 1}V_{i}=\mathcal{P}_{e}(0)$.
For $(x,y)$ in a set $B_{V_{i}}$ of full Lebesgue measure we get from
Theorem $6$ $(1)$ and $(18)$,
\[
\limsup_{T\rightarrow \infty}d(\mu_{T}(x,y),V_{i})\leq \epsilon,
\]  
for all $\epsilon >0$. Thus $\limsup_{T\rightarrow
  \infty}d(\mu_{T}(x,y),V_{i})=0$. Therefore, if $\mu(x,y)$ is a weak limit of
  $\mu_{T}(x,y)$, $(x,y)\in B_{V_{i}}$, we will have $d(\mu(x,y),V_{i})=0$. 
Also, there exists a set $B$ of full Lebesgue measure where we have
  $d(\mu(x,y),V_{i})=0$ for all $i\geq 1$. From this we deduce that
  $\mu(x,y) \in \mathcal{P}_{e}(0)$ for $(x,y)\in B$.
\end{proof}

\subsection{Proof of Theorem $7$}
\subsubsection{Proof of Part $(1)$}
\begin{proof}
The proof adapts the arguments of the proof of Theorem $5$ $(1)$.
Set for any $E\subset \mathcal{P}(SM)$,
\[
\nu_{T}(E)=\frac{\#\{\gamma_{xy}\in G_{T}:\delta_{\gamma_{xy}}\in E\}} 
{\# G_{T}}.
\]
Recall that the functional $Q$ corresponding to the potential $F=0$ is
given by $Q(\omega)=P(\omega)-h_{top}$.

Let $\epsilon >0$. There exists a finite number of continuous
functions $\omega_{1}, \cdots, \omega_{l}$ such that $K\subset
\cup_{i=1}^{l}K_{i}$, where 
\[
K_{i}=\{m\in \mathcal{P}(SM): \int \omega_{i}dm-Q(\omega_{i})>J(K)-\epsilon\}.
\]
Again we can suppose that all the $\omega_{i}$'s are non negative.

Put
\[
\Gamma_{i}(x,y,T):=\{\gamma_{xy}\in
G_{T}:\delta_{\gamma_{xy}}\in K_{i}\},  
\]and
\[
Z_{i}(x,y,T):=\# \Gamma_{i}(x,y,T). 
\]
From the definition of $\Gamma_{i}(x,y,T)$ we get,
\[
Z_{i}(x,y,T)\leq \sum_{\gamma_{xy}\in \Gamma_{i}(x,y,T)}
e^{l(\gamma_{xy})(\int
  \omega_{i}d\delta_{\gamma_{xy}}-
Q(\omega_{i})-(J(K)-\epsilon))}.
\]
Set 
\[
\Gamma_{ij}(x,y,T):=\{\gamma_{xy}:T-(j+1)\delta<l(\gamma_{xy}) \leq
T-j\delta, \  \delta_{\gamma_{xy}}\in K_{i}\}.  
\]
Since the functions $\omega_{i}$ were supposed non negative, we have
$Q(\omega_{i})\geq 0$ and then $Q(\omega_{i})+J(K) \geq 0$. Thus, from
Theorem $9$ $(1)$ and for $T$ sufficiently large,  
\begin{eqnarray*}
Z_{i}(x,y,T)&\leq& \sum_{\gamma_{xy}\in \Gamma_{i}(x,y,T)}
e^{\int_{\gamma_{xy}}\omega_{i}}e^{l(\gamma_{xy}) \left ( 
-Q(\omega_{i})-(J(K)-\epsilon)\right )}\\
&\leq& \sum_{j}e^{(T-(j+1)\delta)\left (-Q(\omega_{i})-(J(K)-\epsilon)\right )}
\sum_{\gamma_{xy}\in \Gamma_{ij}(x,y,T)}
e^{\int_{\gamma_{xy}}\omega_{i}}\\
&\leq& \sum_{j} e^{(T-(j+1)\delta)\left
(-Q(\omega_{i})-(J(K)-\epsilon)\right )} 
e^{(T-j\delta)(P(\omega_{i})+\epsilon)}\\
&=& \sum_{j}e^{(T-j\delta)(h_{top}-J(K)+2\epsilon)}
e^{-\delta(-Q(\omega_{i})-(J(K)-\epsilon))}.
\end{eqnarray*}
We assumed that $J(K)< h_{top}$, then
$C:=\sum_{j}e^{-j\delta (h_{top}-J(K)+2\epsilon)}<\infty$. Setting
$\lambda_{i}:=e^{-\delta(-Q(\omega_{i})-(J(K)-\epsilon))}$, we get
for $T$ sufficiently large and Theorem $8$ $(2)$,
\begin{eqnarray*}
\nu_{T}(K) &\leq& 
\sum_{i=1}^{l} \frac{Z_{i}(x,y,T)}{\# G_{T}}\\
&\leq&
\sum_{i=1}^{l}\sum_{j}
e^{T(-J(K)+3\epsilon)}\lambda_{i}e^{-j\delta (h_{top}-J(K)+2\epsilon)}\\
&=&e^{T(-J(K)+3\epsilon)}C\sum_{i=1}^{l}\lambda_{i} . 
\end{eqnarray*}
Take the logarithme, divide by $T$ and take the $\limsup$,
\[
\limsup_{T\rightarrow \infty}\frac{1}{T}\log \nu_{T}(K) \leq -J(K)+3\epsilon.
\]
$\epsilon$ being arbitrary, this proves Theorem $7$ $(1)$.
\end{proof}
\subsubsection{Proof of Part $(2)$}
The proof follows readily from the proof of
Theorem $6$ $(2)$ and Theorem $7$ $(1)$.

\end{document}